\documentclass[a4paper,11pt]{article}
\usepackage[utf8]{inputenc}

\usepackage{graphicx,amssymb,amsmath,amsfonts,amsthm,hyperref,textcomp,relsize}

\newtheorem{proposition}{Proposition}

\newtheorem{lemma}{Lemma}

\newtheorem{theorem}{Theorem}

\theoremstyle{remark}

\newcommand{\CC}{\mathbb C}

\newcommand{\QQ}{\mathbb Q}
\newcommand{\Fp}{{\mathbb F}_p}
\newcommand{\Fq}{{\mathbb F}_q}
\newcommand{\Fqr}{{\mathbb F}_{q^r}}
\newcommand{\Fpr}{{\mathbb F}_{p^r}}

\newcommand{\Fqm}{{\mathbb F}_{q^m}}

\newcommand{\Ql}{\bar{\mathbb Q}_\ell}
\newcommand{\Dbc}{D^b_c}

\newcommand{\FF}{\mathcal F}
\newcommand{\GGG}{\mathcal G}
\newcommand{\HH}{\mathcal H}

\newcommand{\Gm}{{\mathbb G}_m}

\newcommand{\PP}{\mathbb P}
\newcommand{\HHH}{\mathrm H}

\title{An effective criterion for finite monodromy of $\ell$-adic sheaves}
\author{Antonio Rojas Le\'on \\
        Departamento de \'Algebra \\
        Universidad de Sevilla \\
        c/Tarfia, s/n, 41012 Sevilla, SPAIN \\
        ORCID 0000-0003-1683-9487 \\
        \tt{arojas@us.es}}

\begin{document}

\maketitle


\abstract{We provide an effective version of Katz' criterion for finiteness of the monodromy group of a lisse, pure of weight zero, $\ell$-adic sheaf on a normal variety over a finite field, depending on the numerical complexity of the sheaf.}

\section{Introduction}
Let $X$ be a normal, geometrically irreducible variety over a field $k=\Fq$ of characteristic $p$. Fix a prime $\ell\neq p$, and consider the category $\mathcal S(X,\Ql)$ of lisse $\Ql$ sheaves on $X$. A sheaf will be said to be pure (of a certain weight) if it is so for every embedding $\Ql\to\CC$.

Fix a geometric generic point $\bar\eta$ of $X$. Every sheaf $\FF\in\mathcal S(X,\Ql)$ of rank $r$ corresponds to a continuous representation $\pi_1(X,\bar\eta)\to \mathrm{GL}(r,\Ql)$. The Zariski closure of its image (respectively of the image of the subgroup $\pi_1(\overline X,\bar\eta)$, where $\overline X=X\otimes_{\Fq}\overline{\Fq}$) is called the \emph{arithmetic monodromy group} (resp. the \emph{geometric monodromy group}) of $\FF$. It is well known that, under certain conditions, these groups govern the distribution of the Frobenius traces of the sheaf $\FF$ on the set of rational points of $X$ over larger and larger extensions of $\Fq$.

Several methods have been used in the literature to determine these groups, such as using local monodromies to deduce enough properties of it so that there is only one possibility \cite[Chapter 11]{katz1988gauss}, or computing their moments and applying Larsen's alternative \cite[Chapter 2]{katz2005mma}. In some cases, one can determine that there are only two options: either the group is finite of it is one specific group \cite{such2000monodromy}. In any case, the question of determining whether the monodromy groups are finite is an interesting one.

In general, we have the following criterion, proven in \cite[8.14]{katz1990esa} for curves and in \cite[Proposition 2.1]{krlt-co3} for higher dimensional varieties:
\begin{proposition}\label{old}
 Suppose that $X$ is smooth and $\FF$ is geometrically irreducible and pure of weight $0$. Then the following conditions are equivalent:
 \begin{enumerate}
  \item The geometric monodromy group $G_{geom}$ of $\FF$ is finite.
  \item The arithmetic monodromy group $G_{arith}$ of $\FF$ is finite.
  \item For every finite extension $\Fq\subseteq\Fqr$ and every $t\in X(\Fqr)$, the Frobenius trace of $\FF$ at $t$ is an algebraic integer.
  \item For every finite extension $\Fq\subseteq\Fqr$ and every $t\in X(\Fqr)$, the Frobenius eigenvalues of $\FF$ at $t$ are roots of unity.
 \end{enumerate}
\end{proposition}

 However, this criterion does not provide an effective algorithm, since one needs to check that condition (3) or (4) holds for \emph{every} finite extension of $\Fq$. In practice, one usually needs to resort to ad-hoc methods to check that these conditions hold for a particular $\FF$ (see e.g. \cite[Theorem 3.1]{krlt-co3}). In this article, we prove the following effective version of the criterion, relying on the complexity of $\FF$ as defined by Sawin \cite[Definition 6.4]{quantitative}.

\begin{theorem}
 Suppose that $X$ is smooth and given with a projective embedding $u:X\hookrightarrow{\mathbb P}^n$. Then there exist explicit constants $N_1=N_1(u,d,r,C)$ and $N_2=N_2(u,d,r,C)$ such that for every geometrically irreducible $\FF\in{\mathcal S}(X,\Ql)$ of rank $r$, pure of weight $0$, with complexity $c_u(\FF)\leq C$ and $E$-defined, where $\mathbb Q\subseteq E$ is a finite extension of degree $\leq d$, the following conditions are equivalent:
 \begin{enumerate}
  \item The geometric monodromy group $G_{geom}$ of $\FF$ is finite.
  \item The arithmetic monodromy group $G_{arith}$ of $\FF$ is finite.
  \item For every finite extension $\Fq\subseteq\Fqr$ of degree $\leq N_1$ and every $t\in X(\Fqr)$, the Frobenius trace of $\FF$ at $t$ is an algebraic integer.
  \item For every finite extension $\Fq\subseteq\Fqr$ of degree $\leq N_2$ and every $t\in X(\Fqr)$, the Frobenius eigenvalues of $\FF$ at $t$ are roots of unity.
 \end{enumerate}
\end{theorem}

If $X$ is a smooth curve, we have a more explicit version which does not make use of the general concept of complexity:
\begin{theorem}
 Suppose that $X$ is a smooth curve, $Y$ its smooth projective closure and $D=Y\backslash X$. Then there exist explicit constants $N_1=N_1(X,d,r,e)$ and $N_2=N_2(X,d,r,e)$ such that for every geometrically irreducible $\FF\in{\mathcal S}(X,\Ql)$ of rank $r$, pure of weight $0$, all whose breaks at every $t\in D(\overline Fq)$ are $\leq e$, and $E$-defined, where $\mathbb Q\subseteq E$ is a finite extension of degree $\leq d$, the following conditions are equivalent:
 \begin{enumerate}
  \item The geometric monodromy group $G_{geom}$ of $\FF$ is finite.
  \item The arithmetic monodromy group $G_{arith}$ of $\FF$ is finite.
  \item For every finite extension $\Fq\subseteq\Fqr$ of degree $\leq N_1$ and every $t\in X(\Fqr)$, the Frobenius trace of $\FF$ at $t$ is an algebraic integer.
  \item For every finite extension $\Fq\subseteq\Fqr$ of degree $\leq N_2$ and every $t\in X(\Fqr)$, the Frobenius eigenvalues of $\FF$ at $t$ are roots of unity.
 \end{enumerate}
\end{theorem}

Unfortunately, the explicit constants $N_1$ and $N_2$ we obtain in these theorems are still too large to be useful in practice, so at the moment these results are mainly of theoretical interest. See section \ref{examples} for some numerical examples and some potential ways to optimize them.

\section{$\ell$-adic sheaves and trace functions}
Let $X$ be as in the previous section. Every sheaf $\FF\in\mathcal S(X,\Ql)$ defines a trace function $\Phi_\FF:\coprod_{m\geq 1}X(k_m)\to\Ql$ (where $k_m$ denotes the degree $m$ extension of $k$) given by
$$
\Phi_\FF(m,t)=\mathrm{Tr}(Frob_t|\FF_{\bar t})
$$
where $t\in X(k_m)$ and $\bar t$ is a geometric point over $t$.
By Chevotarev's density theorem, two semisimple sheaves $\FF,\GGG\in\mathcal S(X,\Ql)$ are isomorphic if and only if $\Phi_\FF=\Phi_\GGG$.

For every $m\geq 1$, denote by $\Phi_{\FF,m}:X(k_m)\to\Ql$ the restriction of $\Phi_\FF$ to $X(k_m)$. For a smooth curve $X$, Deligne proved \cite{deligne-finitude} the following bounded version of the previous statement. Let $Y$ be a smooth compactification of $X$, and $D=Y\backslash X$. For every $s\in D(\bar k)$ and $\FF\in\mathcal S(X,\Ql)$, let $\alpha_s(\FF)$ be the largest break of $\FF$ at $s$. Then we have
\begin{theorem}\cite[Proposition 2.5]{deligne-finitude}\label{thm-deligne}
 Let $\FF,\GGG\in \mathcal S(X,\Ql)$ be two lisse semisimple sheaves of rank $r$. Then $\FF$ and $\GGG$ are isomorphic if and only if $\Phi_{\FF,m}=\Phi_{\GGG,m}$ for every $m\leq N$, where
 $$
 N=2r+\left\lfloor 2\log^+_q\left(2r^2\left(b_1(X)+\max_{s\in D(\bar k)}\sup\{\alpha_s(\FF),\alpha_s(\GGG)\}\right)\right)\right\rfloor,
 $$
 $\log_q^+=\max\{0,\log_q\}$ and $b_1(X)=\dim\HHH^1_c(X,\Ql)$ is the first Betti number with compact supports of $X$.
\end{theorem}

We will generalize this to higher dimensional varieties, using Sawin's complexity theory \cite{quantitative} as a replacement for the ramification data. Assume that $X$ is geometrically irreducible, smooth quasi-projective of dimension $d$ over $k$, given with an embedding $u:X\hookrightarrow\PP^n_k$. See \cite[Definition 3.2 and 6.4]{quantitative} for the definition of the complexity $c_u(\FF)\in\mathbb N$ of a lisse sheaf $\FF\in\mathcal S(X,\Ql)$ (or, more generally, an object of $\Dbc(X,\Ql)$). By \cite[Theorem 5.2]{quantitative} there is an explicit constant $A_n$ such that $c_u(\FF\otimes\GGG)\leq A_n c_u(\FF)c_u(\GGG)$ for any $\FF,\GGG\in\mathcal S(X,\Ql)$. In fact, by \cite[Theorem 8.1]{quantitative}, one may take
$$
A_n=\frac{2^{17}}{3^4}e^{4/3}13^n(n+2)!
$$

We then have the following generalization of Theorem \ref{thm-deligne}:

\begin{theorem}\label{traces}
 Let $\FF,\GGG\in \mathcal S(X,\Ql)$ be two lisse semisimple sheaves of rank $r$ of complexity $\leq C$. Then $\FF$ and $\GGG$ are isomorphic if and only if $\Phi_{\FF,m}=\Phi_{\GGG,m}$ for every $m\leq N$, where
 $$
 N=2r+\lfloor 2\log_q^+(2A_n  C^2)\rfloor.
 $$
\end{theorem}

Note that the smoothness hypotheses does not represent a loss of generality, since $\FF$ and $\GGG$ are isomorphic if and only if their restrictions to a dense open smooth $U\subseteq X$ are.

The proof follows closely that of \cite[Proposition 2.5]{deligne-finitude}. As there, we can decompose $\FF$ and $\GGG$ as direct sums
$$
\FF\cong\bigoplus_{i\in I}p_{i\ast}({\mathcal H}_i\otimes{\mathcal W}_i)
$$
$$
\GGG\cong\bigoplus_{i\in I}p_{i\ast}({\mathcal H}_i\otimes{\mathcal W'}_i)
$$
where, for every $i$, there is some $n_i\geq 1$ such that ${\mathcal H}_i$ is a geometrically irreducible lisse sheaf on $X_{n_i}:=X\otimes_k k_{n_i}$ with determinant of finite order, $\mathcal W_i$ and $\mathcal W'_i$ are geometrically constant on $X_{n_i}$ (and at least one of them is non-zero), $p_i:X_{n_i}\to X$ is the natural projection, and the ${\mathcal H}_i$ and their Galois conjugates are pairwise geometrically non-isomorphic.

Let $\mathcal H_{i,j}$ for $0\leq j < n_i$ denote the Frobenius conjugates of $\mathcal H_i$. For every $m\geq 1$, let $I_m$ be the set $\{i\in I: n_i|m\}$. As in \cite[Lemme 2.6]{deligne-finitude}, we can show:

\begin{lemma}
 Let $C=\max\{c_u(\FF),c_u(\GGG)\}$ and $N_0=2\log_q^+(2A_n C^2)$. For $m>N_0$, the functions $\Phi_{\mathcal H_{i,j},m}:X(k_m)\to\Ql$ ($i\in I_m$, $0\leq j < n_i$) are linearly independent. 
\end{lemma}

Let $\sum_{i,j}\lambda_{i,j}\Phi_{\mathcal H_{i,j},m}=0$ be a non-trivial linear combination, and assume without loss of generality that $\lambda_{i_0,j_0}=1$ and $|\lambda_{i,j}|\leq 1$ for every $i,j$. We may also assume that $\mathcal H_{i_0,j_0}$ is a direct summand of $\FF$, by interchanging $\FF$ and $\GGG$ if necessary. Then
$$
0=\sum_{t\in X(k_m)}\left(\sum_{i,j}\lambda_{i,j}\Phi_{\mathcal H_{i,j},m}(t)\right)\overline{\Phi_{\mathcal H_{i_0,j_0},m}(t)}=
$$
$$
=\sum_{i,j}\lambda_{i,j}\sum_{t\in X(k_m)}\Phi_{\mathcal H_{i,j},m}(t)\overline{\Phi_{\mathcal H_{i_0,j_0},m}(t)}=
\sum_{i,j}\lambda_{i,j}\sum_{t\in X(k_m)}\Phi_{\mathcal H_{i,j}\otimes\widehat{\mathcal H_{i_0,j_0}},m}(t)=
$$
$$
=\sum_{i,j}\lambda_{i,j}\sum_{a=0}^{2d}(-1)^a\mathrm{Tr}(F_{k_m}|\HHH^a_c(X\otimes\bar k,\mathcal H_{i,j}\otimes\widehat{\mathcal H_{i_0,j_0}})).
$$
Since the $\HH_{i,j}$ are geometrically irreducible and pairwise non-isomorphic, we have $\HHH^{2d}_c(X\otimes\bar k,\mathcal H_{i,j}\otimes\widehat{\mathcal H_{i_0,j_0}})=0$ if $(i,j)\neq (i_0,j_0)$ and $\Ql(-d)$ if $(i,j)=(i_0,j_0)$, so
$$
0=q^{md}+\sum_{i,j}\lambda_{i,j}\sum_{a=0}^{2d-1}(-1)^a\mathrm{Tr}(F_{k_m}|\HHH^a_c(X\otimes\bar k,\mathcal H_{i,j}\otimes\widehat{\mathcal H_{i_0,j_0}}))
$$
and
$$
q^{md}=\left|\sum_{i,j}\lambda_{i,j}\sum_{a=0}^{2d-1}(-1)^a\mathrm{Tr}(F_{k_m}|\HHH^a_c(X\otimes\bar k,\mathcal H_{i,j}\otimes\widehat{\mathcal H_{i_0,j_0}}))\right|\leq
$$
$$
\leq \sum_{i,j}\sum_{a=0}^{2d-1}|\mathrm{Tr}(F_{k_m}|\HHH^a_c(X\otimes\bar k,\mathcal H_{i,j}\otimes\widehat{\mathcal H_{i_0,j_0}}))|\leq
$$
$$
\leq q^{md-\frac{m}{2}}\sum_{i,j}\sum_{a=0}^{2d-1}\dim\,\HHH^a_c(X\otimes\bar k,\mathcal H_{i,j}\otimes\widehat{\mathcal H_{i_0,j_0}})\leq
$$
\begin{equation}\label{ineq1}
\leq q^{md-\frac{m}{2}}\sum_{a=0}^{2d-1}\dim\,\HHH^a_c(X\otimes\bar k,\mathcal (\FF\oplus\GGG)\otimes\widehat{\FF})
\end{equation}
since $\oplus_{i,j}\mathcal H_{i,j}$ and ${\mathcal H_{i_0,j_0}}$ are direct summands of $\FF\oplus\GGG$ and $\FF$ respectively.

By the properties of the complexity \cite[Theorem 5.1, Lemma 6.12, Theorem 8.1]{quantitative}, we have
$$
\sum_{a=0}^{2d-1}\dim\,\HHH^a_c(X\otimes\bar k,\mathcal (\FF\oplus\GGG)\otimes\widehat{\FF})\leq
$$
$$
\leq\sum_{a=0}^{2d}\dim\,\HHH^a_c(X\otimes\bar k,\mathcal \FF\otimes\widehat{\FF})+\sum_{a=0}^{2d}\dim\,\HHH^a_c(X\otimes\bar k,\mathcal \GGG\otimes\widehat{\FF})\leq
$$
$$
\leq A_n(c_u(\FF)^2+c_u(\FF)c_u(\GGG))\leq 2 A_n C^2.
$$
So, from \ref{ineq1}, we deduce
$$
q^{m/2}\leq 2A_n C^2
$$
or
$$
m\leq 2\log_q(2A_n C^2)
$$

The proof of Theorem \ref{traces} now concludes exactly as in \cite[2.8]{deligne-finitude}: for every $m>N_0$, we have
$$
\Phi_{\FF,m}=\sum_{i\in I_m}\sum_{j=1}^{n_i}\mathrm{Tr}(F_{k_m}|\mathcal W_i)\Phi_{\HH_{i,j},m}
$$
and
$$
\Phi_{\GGG,m}=\sum_{i\in I_m}\sum_{j=1}^{n_i}\mathrm{Tr}(F_{k_m}|\mathcal W'_i)\Phi_{\HH_{i,j},m}
$$
so, by the lemma, $\Phi_{\FF,m}=\Phi_{\GGG,m}$ if and only if $\mathrm{Tr}(F_{k_m}|\mathcal W_i)=\mathrm{Tr}(F_{k_m}|\mathcal W'_i)$ for every $i\in I_m$. For every $i\in I$, $\mathcal W_i$ and $\mathcal W'_i$ have dimension at most $\lfloor r/n_i\rfloor$ so, in order to show they are isomorphic (that is, that they have the same eigenvalues for the action of $F_{k_{n_i}}$), it suffices to show that the traces of the action of $F_{k_{n_i}}^l=F_{k_{ln_i}}$ on them coincide for $\lfloor 2r/n_i\rfloor$ consecutive values of $l$ \cite[Lemme 2.9]{deligne-finitude}. But, under the hypotheses of Theorem \ref{traces}, these traces coincide for $N_0<ln_i\leq N_0+2r$, in which there are at least $\lfloor 2r/n_i\rfloor$ possible values of $l$. 

\section{An effective criterion for finite monodromy}

For a finite extension $E$ of $\mathbb Q$ and a positive integer $r$, let $M(E,r)$ denote the least common multiple of the $n\geq 1$ such that $[E(\zeta_n):E]\leq r$, where $\zeta_n$ is a primitive $n$-th root of unity. For instance, we have
$$
M(\mathbb Q,r)=\prod_{\lambda\text{ prime}\leq r+1}\lambda^{\lfloor 1+\log_\lambda\frac{r}{\lambda-1}\rfloor}
$$
and, in general, $M(E,r)\leq M(\mathbb Q,r\cdot[E:\mathbb Q])$, since
$$
[E(\zeta_n):E]\leq r\Rightarrow [\mathbb Q(\zeta_n):\mathbb Q]\leq [E(\zeta_n):\mathbb Q]\leq r\cdot [E:\mathbb Q].
$$

Given an $\HH\in \mathcal S(X,\Ql)$, we say that $\HH$ is $E$-valued if, for every $m\geq 1$ and $x\in U(\Fqm)$, the characteristic polynomial of the Frobenius action on $\HH$ at $x$ has coefficients in $E$. By \cite[Th\'eor\`eme VII.6]{lafforgue}, if $\HH$ is irreducible and its determinant is arithmetically of finite order then it is $E$-valued for some finite extension $\mathbb Q\subseteq E$.

\begin{proposition}\label{prop2}
 Let $\HH$ be a geometrically irreducible lisse $\ell$-adic sheaf on $X$ of rank $r$, pure of weight $0$ and whose determinant is aritmetically of finite order, let $\mathbb Q\subseteq E$ be a finite extension such that $\HH$ is $E$-valued, and $M=M(E,r)$. Let
 
  $$
 N:=2R+ \lfloor 2\log_q^+(2 A_n (A_n^{M-1} C^M+r\cdot c_u(X))^2)\rfloor
 $$
 where
 $$
 R=\sum_{\stackrel{i=0}{i\text{ even}}}^{r-1}{{r+M-i-1}\choose{M}}{{M-1}\choose i}
 $$
%
%
 Then $\HH$ has finite (arithmetic and geometric) monodromy if and only if all Frobenius eigenvalues of $\HH$ at $t$ are roots of unity for every $m\leq N$ and every $x\in X(\Fqm)$.
\end{proposition}

\begin{proof}
Suppose that the Frobenius eigenvalues of $\HH$ at $x$ are roots of unity for every $m\leq N$ and every $x\in X(\Fqm)$. Since these eigenvalues are roots of a polynomial of degree $r$ over $E$, their order divides $M$ by definition. That is, the $M$-th power of Frobenius acts trivially on $\HH_{\bar x}$ for every $x\in X(\Fqm)$, $m\leq N$.

Let $\HH^{[M]}$ be the $M$-th Adams power of $\HH$. It is an element of the Grothendieck group of the category of constructible sheaves on $X$ and, by \cite[1]{fu2004moment}, its trace function is given by $\Phi_{\HH^{[M]}}(m,x)=\Phi_{\HH}(Mm,x)$.
 By \cite[Proposition 3.4]{rojas-rationality}, we have the explicit expression $\HH^{[M]}=\sum_{i=0}^{M-1}(-1)^i[R(\wedge^i\rho)\HH]$, where $\rho$ is the standard $(M-1)$-dimensional representation of the symmetric group $\mathcal S_M$. By the previous paragraph, we have
$$
\Phi_{\HH^{[M]}}(m,x)=r=\Phi_{\Ql^n}(m,x)
$$
for every $m\leq N$ and $x\in X(\Fqm)$. This can be rewritten in terms of ``real'' sheaves by splitting the positive and negative components of $\HH^{[M]}$: let $\FF=\oplus_{i\text{ even}}R(\wedge^i\rho)\HH$ and $\GGG=\oplus_{i\text{ odd}}R(\wedge^i\rho)\HH$, then
$$
\Phi_{\FF}(m,x)=\Phi_{\GGG\oplus\Ql^n}(m,x)
$$
for every $m\leq N$ and $x\in X(\Fqm)$.

Since $\FF$ and $\GGG$ are subsheaves of $\HH^{\otimes M}$, their complexity is bounded by that of $\HH^{\otimes M}$, which in turn, applying \cite[Theorem 5.2]{quantitative} repeatedly, is bounded by $A_n^{M-1} C^M$. Then the complexity of $\GGG\oplus\Ql^r$ is bounded by $A_n^{M-1} C^M+r\cdot c_u(X)$.

Since $\FF$ and $\GGG\oplus\Ql^r$ have rank $\sum_{i\text{ even}}{{r+M-i-1}\choose{M}}{{M-1}\choose i}$ \cite[Remark 3.6]{rojas-rationality} and are pure of weight 0, by Theorem \ref{traces} we conclude that $\FF\cong\GGG\oplus\Ql^n$ or, equivalently, $\HH^{[M]}=[\FF]-[\GGG]=[\Ql^n]$ in the Grothendieck group. That is, the $M$-th power of Frobenius acts trivially on $\HH_{\bar x}$ for every $x\in X(\Fqm)$ and every $m\geq 1$. By Proposition \ref{old}, we conclude that $\HH$ has finite monodromy.
\end{proof}

If $X$ is a smooth curve, then we can improve the bound by using Theorem \ref{thm-deligne} instead of Theorem \ref{traces}. Let $Y$ be the smooth projective closure of $X$ and $D:=Y\backslash X$. We then get

\begin{proposition}\label{propcurves}
 Let $\HH$ be a geometrically irreducible lisse $\ell$-adic sheaf on $X$ of rank $r$, pure of weight $0$ and whose determinant is aritmetically of finite order, let $\mathbb Q\subseteq E$ be a finite extension such that $\HH$ is $E$-valued, and $M=M(E,r)$.   
  For every $x\in D(\overline\Fq)$, assume that the breaks of $\HH$ at $x$ are $\leq e_x$, and let $e:=\sum_{x\in D(\overline\Fq)}e_x$ and
$$
 N=2R+\lfloor 2\log^+_q\left(2R^2\left(b_1(X)+e\right)\right)\rfloor,
 $$
where 
$$
R=\sum_{\stackrel{i=0}{i\text{ even}}}^{r-1}{{r+M-i-1}\choose{M}}{{M-1}\choose i}.
$$

 Then $\HH$ has finite (arithmetic and geometric) monodromy if and only if all Frobenius eigenvalues of $\HH$ at $t$ are roots of unity for every $m\leq N$ and every $x\in X(\Fqm)$.
\end{proposition}

\begin{proof}
 The proof goes exactly as in Proposition \ref{prop2}, using Theorem \ref{thm-deligne} and the fact that all breaks of ${\mathcal H}^{\otimes r}$ at $x\in D(\overline\Fq)$ are $\leq e_x$.
\end{proof}

Next, we give similar results based on the integrality of the Frobenius traces instead of its Frobenius eigenvalues, which are generally easier to compute. We start by showing

\begin{lemma}\label{lema2}
 Let $\mathbb Q_p\subseteq E_\pi$ be a finite extension with ramification index $e$, and let $\alpha_1,\ldots,\alpha_r\in E_\pi$. Let $a=\lfloor\log_p r\rfloor$ and
 $$
 N:=r\left(1+\left\lfloor\frac{e}{p-1}\left(1-\frac{1}{p^a}\right)\right\rfloor \right)
 $$
 Then $\alpha_1,\ldots,\alpha_r$ are integral if and only if $\sum_{i=1}^r\alpha_i^k$ is integral for every $1\leq k\leq N$.
\end{lemma}

\begin{proof}
 Let $M=1+\left\lfloor\frac{e}{p-1}\left(1-\frac{1}{p^a}\right)\right\rfloor$; $p_k:=\sum_{i=1}^n\alpha_i^{Mk}$ for $k\geq 1$ and $s_k$ be the $k$-th elementary symmetric function on $\alpha^M_1,\ldots,\alpha^M_r$ for $k=0,\ldots,r$. By Newton's identities, $ks_k=\sum_{i=1}^k(-1)^{i-1}s_{k-i}p_i$, so $k!s_k$ is integral for every $k=0,\ldots,r$.
 
 If $\nu$ denotes the valuation on $E_\pi$, normalized so that $\nu(p)=1$, we get
 $$
 \nu(s_k)\geq -\nu(k!)=-\sum_{i=1}^\infty\left\lfloor\frac{k}{p^i}\right\rfloor.
 $$
 so the Newton polygon of the polynomial $\prod_{i=1}^r(1-\alpha_i^MT)$ is bounded below by the polygon with vertices 
 $$
 \left\{\left(k,-\sum_{i=1}^\infty\left\lfloor\frac{k}{p^i}\right\rfloor\right); k=0,1,\ldots,r\right\}
 $$
 and, in particular, its largest slope (in absolute value) is bounded by
 $$
 \frac{1+p+\cdots+p^{a-1}}{p^a}=\frac{1}{p-1}\left(1-\frac{1}{p^a}\right)
 $$
 so $\nu(\alpha_i^M)\geq -\frac{1}{p-1}\left(1-\frac{1}{p^a}\right)$ for every $i$, and
 $$
 \nu(\alpha_i)\geq -\frac{1}{M(p-1)}\left(1-\frac{1}{p^a}\right)>-\frac{1}{e}.
 $$
 But $\nu(E_\pi)=\frac{1}{e}\mathbb Z$, so we conclude that $\nu(\alpha_i)\geq 0$ for every $i=1,\ldots,n$.   
\end{proof}

\begin{theorem}

 Let $\HH$ be a geometrically irreducible lisse $\ell$-adic sheaf on $X$ of rank $r$, pure of weight $0$ and whose determinant is aritmetically of finite order, let $\mathbb Q\subseteq E$ be finite extension such that $\HH$ is $E$-valued, $f$ be the maximum among the ramification indices of the primes of $E$ above $p$, and $a=\lfloor\log_p n\rfloor$. Let 
   $$
 N:=r\left(1+\left\lfloor\frac{rf}{p-1}\left(1-\frac{1}{p^a}\right)\right\rfloor \right) (2R+ \lfloor 2\log_q^+(2A_n (A_n^{M-1} C^M+r\cdot c_u(X))^2)\rfloor )
 $$
 where
 $$
 R=\sum_{\stackrel{i=0}{i\text{ even}}}^{r-1}{{r+M-i-1}\choose{M}}{{M-1}\choose i}
 $$
 
 
 Then $\HH$ has finite (arithmetic and geometric) monodromy if and only if $\Phi_\HH(m,x)$ is integral at all places of $E$ over $p$ for every $m\leq N$ and every $x\in X(\Fqm)$.
\end{theorem}

\begin{proof}
 Assume that all Frobenius traces of $\HH$ at $x$ are integral at all places of $E$ over $p$ for every $m\leq N$ and every $x\in X(\Fqm)$. For every $m\geq 1$ and every $x\in X(\Fqm)$, the Frobenius eigenvalues of $\HH$ at $t$ are contained in some extension of $E$ of degree $\leq r$, in which all primes above $p$ have ramification index $\leq rf$. By lemma \ref{lema2}, all Frobenius eigenvalues of $\HH$ at $t$ are integral for every $m\leq N'$ and every $x\in X(\Fqm)$ at all places over $p$, where
 $$
 N'=2R+ \lfloor 2\log_q(2A_n (A_n^{M-1} C^M+r\cdot c_u(X))^2)\rfloor.
 $$
 
 By \cite[Th\'eor\`eme VII.6]{lafforgue}, they are also integral at all other non-archimedean places, so they are algebraic integers. Since their product is a root of unity, they must all be roots of unity, and we conclude by Proposition \ref{prop2}.
\end{proof}

For $X$ a smooth curve we get the following, more optimized, result by using \ref{propcurves} instead of \ref{prop2}:

\begin{theorem}

 Let $\HH$ be a geometrically irreducible lisse $\ell$-adic sheaf on $X$ of rank $r$, pure of weight $0$ and whose determinant is aritmetically of finite order, let $\mathbb Q\subseteq E$ be finite extension such that $\HH$ is $E$-valued, $f$ be the maximum among the ramification indices of the primes of $E$ above $p$, and $a=\lfloor\log_p r\rfloor$.  For every $x\in D(\overline\Fq)$, assume that the breaks of $\HH$ at $x$ are $\leq e_x$, and let $e:=\sum_{x\in D(\overline\Fq)}e_x$ and
$$
 N=r\left(1+\left\lfloor\frac{rf}{p-1}\left(1-\frac{1}{p^a}\right)\right\rfloor \right) (2R+\lfloor 2\log^+_q\left(2R^2\left(b_1(X)+e\right)\right)\rfloor),
 $$
where 
$$
R=\sum_{\stackrel{i=0}{i\text{ even}}}^{r-1}{{r+M-i-1}\choose{M}}{{M-1}\choose i}.
$$
 
 
 Then $\HH$ has finite (arithmetic and geometric) monodromy if and only if $\Phi_\HH(m,x)$ is integral at all places of $E$ over $p$ for every $m\leq N$ and every $x\in X(\Fqm)$.
\end{theorem}

\section{Examples}
\label{examples}

Let $p$ be a prime, $X=\mathbb A^1_{\Fp}$ the affine line, fix an additive character $\psi:\Fp\to\CC$ and an integer $n\geq 2$, and consider the sheaf $\FF\in{\mathcal S}(X,\Ql)$ whose trace function at $t\in X(\Fpr)$ is given by
$$
-\frac{1}{G^r}\sum_{x\in\Fpr}\psi(\mathrm{Tr}_{\Fpr/\Fp}(x^n+tx))
$$
where $G=-\sum_{x\in\Fp}\psi(x^2)$ is the Gauss sum. That is, the (normalized) Fourier transform of the pull-back of the Artin-Schreier sheaf $\mathcal L_\psi$ by the $n$-th power map. The sheaf $\FF$ is lisse of rank $n-1$, pure of weight $0$ and $\QQ(\zeta_p)$-valued, where $\zeta_p$ is a $p$-th root of unity. We have
$$
M=M(\QQ(\zeta_p),n-1)=p^{\lfloor 1+\log_p(n-1)\rfloor}\prod_{\stackrel{\lambda\leq n \mathrm{ prime}}{\lambda\neq p}}\lambda^{\lfloor 1+\log_\lambda\frac{n-1}{\lambda-1}\rfloor},
$$
$b_1(X)=0$, $e=\frac{1}{n-1}$ and $f=p-1$ (as $p$ is totally ramified in $\QQ(\zeta_p)$), so in this case we get finite monodromy iff all Frobenius eigenvalues are roots of unity at all points defined over extensions of $\Fp$ of degree up to
$$
2R+\left[2\log^+_p\left(2R^2/(n-1)\right)\right],
$$
or if all Frobenius traces are integral at all points over extensions of $\Fp$ of degree up to
$$
(n-1)\left(1+\left\lfloor(n-1)\left(1-\frac{1}{p^{\lfloor \log_p(n-1)\rfloor}}\right)\right\rfloor \right)(2R+\left[2\log^+_p\left(2R^2/(n-1)\right)\right]).
$$

Even for $p=2$, this gives degrees up to 40, 319 and 2304402 respectively for $n=3,4,5$ for the eigenvalues criterion, far beyond what we can compute in practice.

Let now $X={\Gm}_{,\Fq}$ be the one-dimensional torus over $\Fq$ where $q=p^r$, fix an additive character $\psi:\Fp\to\CC$  and two disjoint sets of multiplicative characters $\mathbf{\chi}=\{\chi_1,\ldots,\chi_a\}$ and $\mathbf{\rho}=\{\rho_1,\ldots,\rho_b\}$, and consider the (normalized) hypergeometric sheaf  
$$\HH:=\left(\frac{1}{G^{r(a+b-1)}}\right)^{deg}\HH_1(!;\psi\circ\mathrm{Tr}_{\Fq/\Fp};\chi;\rho)$$
\cite[Chapter 8]{katz1990esa}. Assume $a>b$, then $\HH$ is lisse of rank $a$ on $X$, pure of weight $0$ \cite[Theorem 8.4.2]{katz1990esa}, and $\QQ(\zeta_{pm})$-valued, where $m=\mathrm{lcm}(\mathrm{ord}(\chi_i),\mathrm{ord}(\rho_j))$. Here 
$$
M= M(\QQ(\zeta_{mp}),a)=p^{\lfloor 1+\log_p a\rfloor}\prod_{\lambda|m \mathrm{ prime}}\lambda^{\lfloor \mathrm{ord}_\lambda a+\log_\lambda a\rfloor}
\prod_{\stackrel{\lambda\leq a+1 \mathrm{ prime}}{\lambda\nmid mp}}\lambda^{\lfloor 1+\log_\lambda\frac{a}{\lambda-1}\rfloor},
$$
$b_1(X)=1$, $e=\frac{1}{a-b}$ and $f=p-1$ (as $p$ is totally ramified in $\QQ(\zeta_p)$ and unramified in $\QQ(\zeta_m)$), so in this case we get finite monodromy iff all Frobenius eigenvalues are roots of unity at all points defined over extensions of $\Fq$ of degree up to
$$
2R+\left\lfloor2\log^+_q\left(2R^2\left(1+\frac{1}{a-b}\right)\right)\right\rfloor\leq 2R+\lfloor 2\log^+_q\left(4R^2\right)\rfloor,
$$
or if all Frobenius traces are integral at all points over extensions of $\Fp$ of degree up to
$$
a\left(1+\left\lfloor a\left(1-\frac{1}{p^{\lfloor \log_p a\rfloor}}\right)\right\rfloor \right)(2R+\lfloor 2\log^+_q\left(2R^2(1+1/(a-b))\right)\rfloor)\leq
$$
$$
\leq a\left(1+\left\lfloor a\left(1-\frac{1}{p^{\lfloor \log_p a\rfloor}}\right)\right\rfloor \right)(2R+\lfloor 2\log^+_q(4R^2)\rfloor)
$$

For $p=2$ and $a=2$, this gives degrees up to 54 and 124 respectively for $m=3,5$ for the eigenvalues criterion, which is again beyond what we can compute in practice.

Proposition \ref{prop2} can be improved in particular cases if we can get better estimates for the dimensions of the cohomology groups of the tensor powers of $\FF$, such as the ones given in \cite{katz-betti}.

The main obstacle that makes the bounds for $N$ so large in these theorems is the big rank of the components of $\FF^{[M]}$. It is easy to see that the $N$ in Theorem \ref{traces} can not be made smaller than $O(r)$ - even in dimension $0$, one could take for instance the push-forward of the trivial sheaf from $\mathrm{Spec}\,\Fqr$ to $\mathrm{Spec}\,\Fq$, which has rank $r$ and trace $0$ over $\mathrm F_{q^i}$ for $i<r$. However, in the proof of Proposition \ref{prop2} we have equality not only of the Frobenius traces of $\FF$ and $\GGG\oplus\Ql^r$, but also of their Frobenius characteristic polynomials. Any improvement of Theorem \ref{traces} in the case where we have equality of Frobenius characteristic polynomials, and not just of traces, would lead to similar improvements in the size of the $N$'s in the theorems, which could hopefully make them usable in practice.

\section{Acknowledgements}

The author would like to thank N. Katz and P.H. Tiep for their useful comments on earlier versions of the manuscript, and H. Esnault for her remarks about Deligne's Theorem \ref{thm-deligne}.

The author was partially supported by PID2020-114613GB-I00 (Ministerio de Ciencia e Innovaci\'on), P20-01056 and US-1262169 (Consejer\'{\i}a de Econom\'{\i}a, Conocimiento, Empresas y Universidad, Junta de Andaluc\'{\i}a and FEDER)

%

\bibliographystyle{plain}
\bibliography{bibliography}

\end{document}